\documentclass[12pt,reqno]{amsart}
\usepackage{amsmath,amssymb,amsfonts,amscd,latexsym,amsthm,mathrsfs,verbatim,comment}
\usepackage{graphicx}
\usepackage[unicode]{hyperref}
\usepackage{hypbmsec}
\newtheorem{theorem}{Theorem}
\newtheorem{lemma}{Lemma}
\let\wh\widehat
\let\wt\widetilde

\renewcommand{\d}{{\mathrm d}}

\newcommand{\bc}{{\boldsymbol c}}
\newcommand{\fa}{{\mathfrak a}}
\newcommand{\fb}{{\mathfrak b}}
\newcommand{\fh}{{\mathfrak h}}
\newcommand{\fG}{{\mathfrak G}}
\let\Ga\Gamma
\begin{document}

\title{Hypergeometry inspired~by~irrationality~questions}

\author{Christian Krattenthaler}
\address{Fakult\"at f\"ur Mathematik, Universit\"at Wien, Oskar-Morgenstern-Platz 1, A-1090 Vienna, Austria}
\urladdr{http://www.mat.univie.ac.at/~kratt}

\author{Wadim Zudilin}
\address{Department of Mathematics, IMAPP, Radboud University, PO Box 9010, 6500~GL Nijmegen, Netherlands}
\urladdr{http://www.math.ru.nl/~wzudilin}


\subjclass[2010]{11J72, 11M06, 11Y60, 33C20, 33D15, 33F10}

\keywords{Irrationality; zeta value; $\pi$; Catalan's constant; $\log2$; hypergeometric series}

\thanks{The first author is partially supported by the Austrian Science Foundation FWF, grant S50-N15,
in the framework of the Special Research Program ``Algorithmic and Enumerative Combinatorics''.}

\begin{abstract}
We report new hypergeometric constructions of rational approximations to Catalan's constant, $\log2$, and $\pi^2$,
their connection with already known ones, and underlying `permutation group' structures.
Our principal arithmetic achievement is a new partial irrationality result for the values of Riemann's zeta function at odd integers.
\end{abstract}

\maketitle

\section{Introduction}
\label{s1}

Given a real (presumably irrational!) number $\gamma$, how can one prove that it is irrational?
In certain cases (like for square roots of rationals) this is an easy task.
A more general strategy proceeds by the construction of a sequence of rational approximations $r_n=q_n\gamma-p_n\ne0$
such that $\delta_nq_n$, $\delta_np_n$ are integers for some positive integers $\delta_n$ and
$\delta_nr_n\to0$ as $n\to\infty$. This indeed guarantees that $\gamma$ is not rational.
Usually, as a bonus, such a construction also allows one to estimate the irrationality of~$\gamma$ in a quantitative form.

Producing such a sequence of rational Diophantine approximations, even with a weaker requirement on the growth, like $r_n\to0$ as $n\to\infty$, is a difficult problem.
For certain specific `interesting' numbers $\gamma\in\mathbb R$ such sequences are constructed as values of so-called hypergeometric functions; for related definitions of the latter in the ordinary and basic ($q$-) situations
we refer the reader to the books \cite{Ba35,Sl66,GR04}. One of the underlying mechanisms behind the hypergeometric settings
is the existence of numerous transformations of hypergeometric functions, that is, identities that represent the same numerical (or $q$-) quantity
in different looking ways. An arithmetic significance of such transformations is the production of identities of the form $r_n=\wt r_n$ say, where
$r_n=q_n\gamma-p_n$ and $\wt r_n=\wt q_n\gamma-\wt p_n$ for $n=0,1,2,\dots$,
while an analysis of the asymptotic behaviour of $r_n$ or $\wt r_n$, and of the corresponding (\emph{a priori} different) denominators $\delta_n$ or $\wt\delta_n$ are simpler for one of them than for the other.
In several situations, the machinery can be inverted: the equality $r_n=\wt r_n$ is predicted by computing a number of first approximations,
and then established by demonstrating that both sides satisfy the same linear recursion.
Such instances naturally call for finding purely hypergeometric proofs, which in turn may offer more general forms of the approximations.
It comes as no surprise that our computations below have been carried
out using the \textsl{Mathematica} packages HYP and HYPQ \cite{Kr95}. 

The symbiosis of arithmetic and hypergeometry is the main objective of the present note, with special emphasis on (hypergeometric) rational approximations to the following three mathematical constants (in order of their appearance below):
\begin{itemize}
\item Catalan's constant $G=\displaystyle\sum_{k=0}^\infty\frac{(-1)^k}{(2k+1)^2}$,
\qquad
\raisebox{-4mm}{\hbox{\includegraphics[scale=0.012]{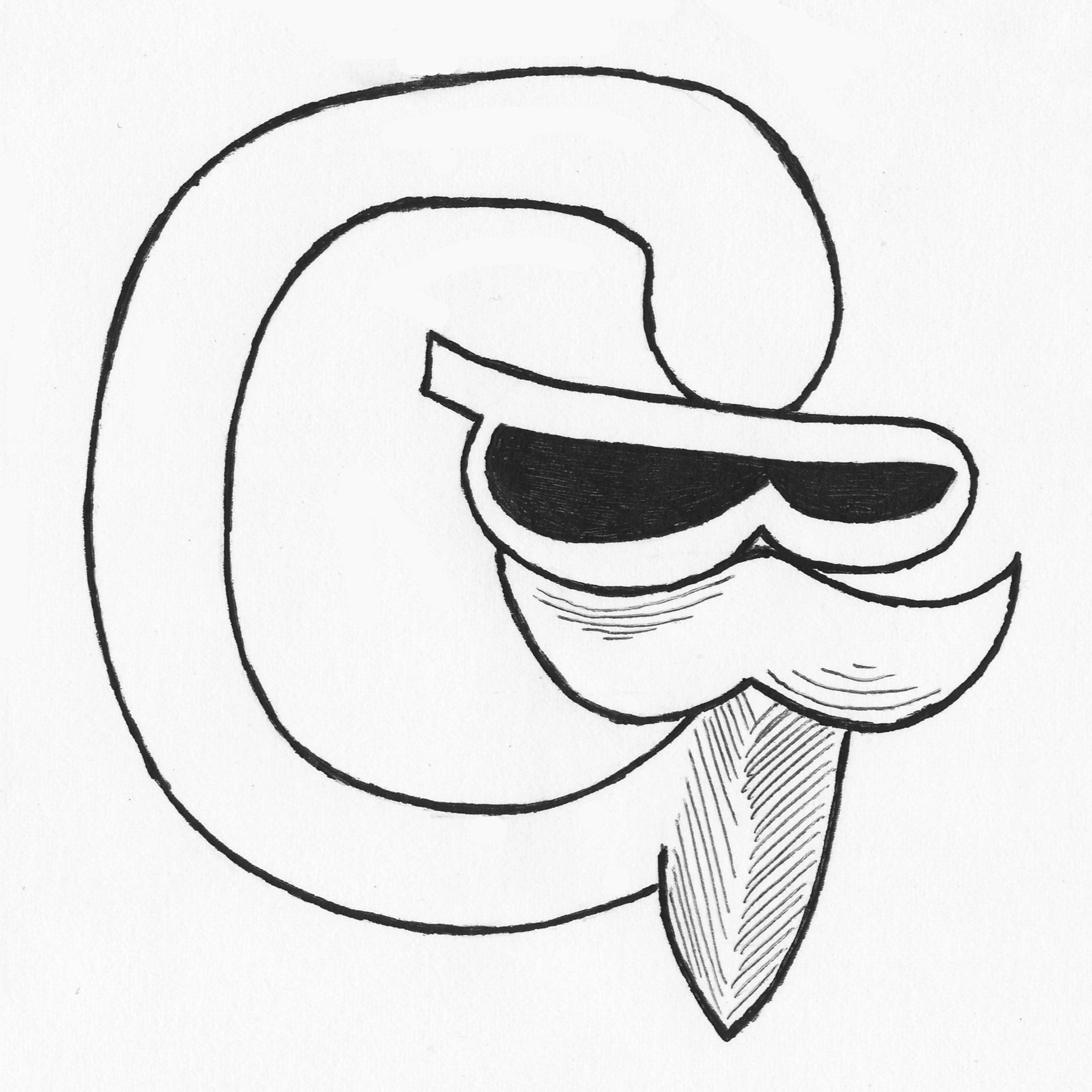}}}
\item $\displaystyle\log2=\sum_{k=1}^\infty\frac{(-1)^{k-1}}{k}$,
\qquad
\raisebox{-4mm}{\hbox{\includegraphics[scale=0.012]{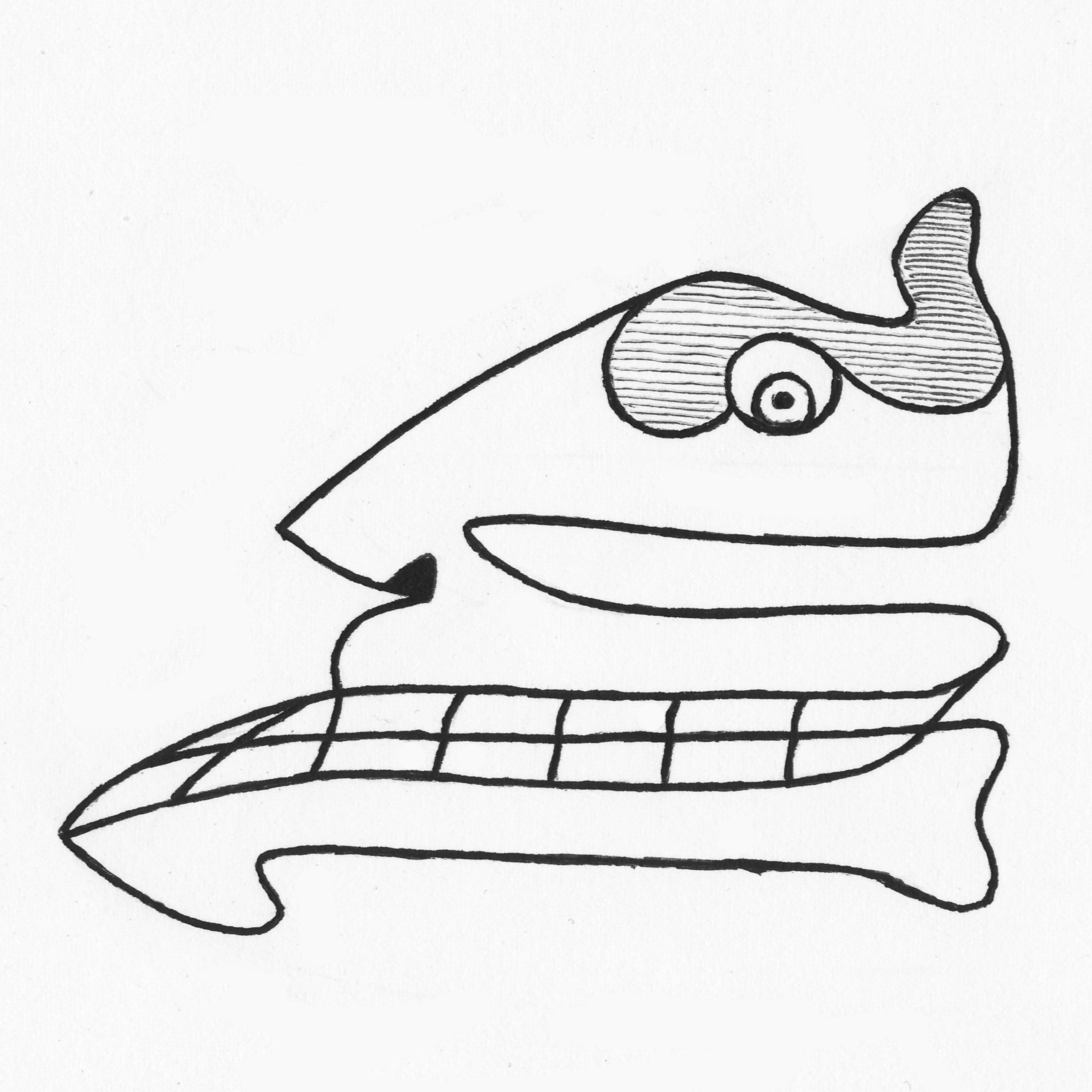}}}
\qquad\qquad
and
\item $\displaystyle\frac{\pi^2}6=\zeta(2)=\sum_{k=1}^\infty\frac1{k^2}$,
\qquad
\raisebox{-4mm}{\hbox{\includegraphics[scale=0.012]{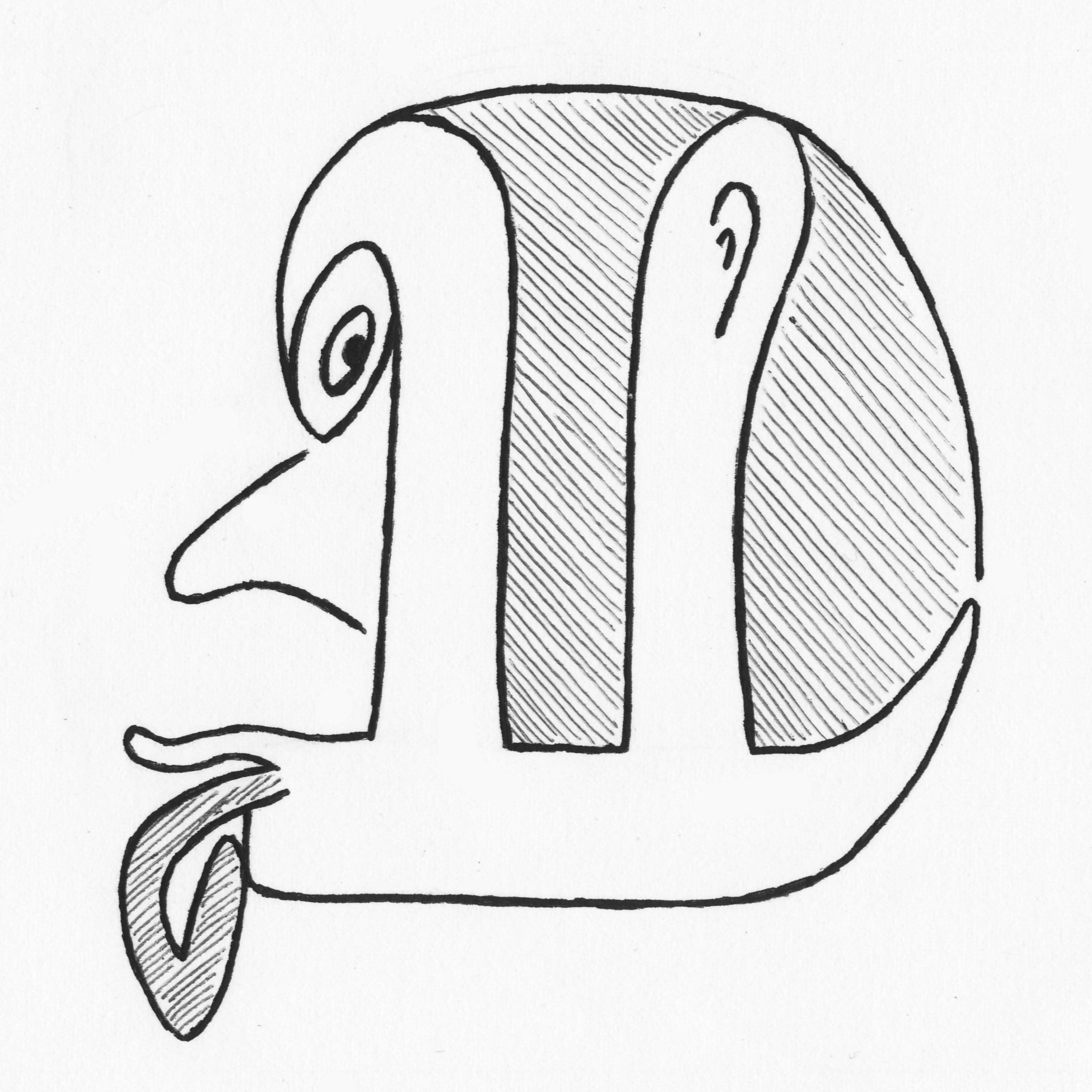}}}
\end{itemize}
which are discussed in Sections \ref{s2}, \ref{s3}, and \ref{s4}, respectively.
We intentionally personify these mathematical constants here, to stress their significance in the arithmetic-hypergeometric context.

The construction in Section~\ref{s4} indicates a certain cancellation phenomenon, which we record in Lemma~\ref{lem1}.
Application of this new ingredient to a general construction
of linear forms in the values of Riemann's zeta function $\zeta(s)$ at positive odd integers leads to the following result.

\begin{theorem}
\label{th-zeta}
For any $\lambda\in\mathbb R$, each of the two collections
$$
\biggl\{\zeta(2m+1)-\lambda\,\frac{2^{2m}(2^{2m+2}-1)|B_{2m+2}|}{(2^{2m+1}-1)(m+1)(2m)!}\,\pi^{2m+1}:m=1,2,\dots,19\biggr\}
$$
and
$$
\biggl\{\zeta(2m+1)-\lambda\,\frac{2^{2m}(2^{2m}-1)|B_{2m}|}{(2^{2m+1}-1)m(2m)!}\,\pi^{2m+1}:m=1,2,\dots,21\biggr\}
$$
contains at least one irrational number.
Here $B_{2m}$ denotes the $2m$-th Bernoulli number.
\end{theorem}

We prove this theorem in Section~\ref{s5}.
Notice that
$$
\frac{2^{2m-1}|B_{2m}|}{(2m)!}
=\frac{\zeta(2m)}{\pi^{2m}}\in\mathbb Q \quad\text{for}\; m=1,2,\dotsc.
$$

The only result in the literature we can compare our Theorem~\ref{th-zeta} with is the one given
in \cite[Theorems~3 and 4]{HP06}, which implies the irrationality of at least one number in each collection
$$
\biggl\{\zeta(2m+1)-\lambda\,\frac{2^{2m}|B_{2m}|}{m(2m)!}\,\pi^{2m+1}:m=1,2,\dots,169\biggr\}
$$
and
$$
\biggl\{\zeta(2m+1)-\lambda\,\frac{2^{2m}|B_{2m+2}|}{(m+1)(2m)!}\,\pi^{2m+1}:m=1,2,\dots,169\biggr\},
$$
where $\lambda\in\mathbb R$ is arbitrary.

\section*{Acknowledgement}
We thank Victor Zudilin for beautifully portraying the mathematical constants involved here.
We kindly acknowledge the referee's very attentive reading of the original version.

\section{Catalan's constant}
\label{s2}

\null\hfill\smash{\raisebox{-12mm\relax}%
  {\includegraphics[scale=0.03]{catalan}}}\strut \\[-15mm]

\parshape 5
  0pt 0.8\textwidth
  0pt 0.8\textwidth
  0pt 0.8\textwidth
  0pt 0.8\textwidth
  0pt \textwidth
A long time ago, in joint work with T.~Rivoal \cite{RZ03},
the second author considered very-well-poised hypergeometric series that represent linear forms in Catalan's and related constants.
The approximations to Catalan's constant itself were given by
\begin{align*}
r_n
&=\sum_{t=0}^\infty(2t+n+1)\frac{n!\prod_{j=1}^n(t+1-j)\prod_{j=1}^n(t+n+j)}{\prod_{j=0}^n(t+j+\frac12)^3}\,(-1)^{n+t}
\\
&=\frac{\sqrt\pi\,\Gamma(3n+2)\,\Gamma(n+\frac12)^2\Gamma(n+1)}{4^n\,\Gamma(2n+\frac32)^3}
\\ &\qquad\times
{}_6F_5\biggl[\begin{matrix} 3n+1, \, \frac{3n}2+\frac32, \, n+\frac12, \, n+\frac12, \, n+\frac12, \, n+1 \\[2pt]
\frac{3n}2+\frac12, \, 2n+\frac32, \, 2n+\frac32, \, 2n+\frac32, \, 2n+1 \end{matrix}; -1\biggr].
\end{align*}
The approximations possess different hypergeometric forms, for example,
as a $_3F_2(1)$-series and as a Barnes-type integral as discussed in \cite{Zu02b} and \cite{Zu03}.

The use of partial-fraction decomposition in \cite{Zu02b} suggests considering a different family of approximations:
\begin{align*}
\wt r_n
&=2^{2(n+1)}\sum_{t=1}^\infty(2t-1)\frac{(2n+1)!\prod_{j=0}^{2n-1}(t-n+j)}{\prod_{j=0}^{2n+1}(2t-n-\frac32+j)^2}
\\
&=\frac{2^{2(n+1)}\Gamma(2n+2)^2\Gamma(n+\frac12)^2}{\Gamma(3n+\frac52)^2}
\\ &\qquad\times
{}_6F_5\biggl[\begin{matrix} 2n+1, \, n+\frac32, \, \frac n2+\frac14, \, \frac n2+\frac14, \, \frac n2+\frac34, \, \frac n2+\frac34 \\[2pt]
n+\frac12, \, \frac{3n}2+\frac74, \, \frac{3n}2+\frac74, \, \frac{3n}2+\frac54, \, \frac{3n}2+\frac54 \end{matrix}; 1\biggr].
\end{align*}
This is again a very-well-poised ${}_6F_5$-series, but this time evaluated at 1. In addition, it is reasonably easy to show that
$2^{4n}d_{2n-1}^2\wt r_n\in\mathbb Z+\mathbb Z\,G$, 
where $d_N$ denotes the least common multiple of $1,\dots,N$,
using an argument similar to the one in \cite{Zu02b}.

Amazingly, we have $r_n=\wt r_n$, which accidentally came out of the recursion satisfied by $\wt r_n$.
Our first result is a general identity, of which the equality is a special case
(namely, $c=d=n+\frac12$).

\begin{theorem}
\label{th-cat}
We have
\begin{align}
&
\,{}_6F_5\biggl[\begin{matrix} 3n+1, \, \frac{3n}2+\frac32, \, n+\frac12, \, n+1, \, c, \, d \\[2pt]
\frac{3n}2+\frac12, \, 2n+\frac32, \, 2n+1, \, 3n+2-c, \, 3n+2-d \end{matrix}\,; -1\biggr]
\nonumber\\ &\quad
=\frac{\Gamma(4n+3)\,\Gamma(3n+2-c)\,\Gamma(3n+2-d)\,\Gamma(4n+3-c-d)}{\Gamma(3n+2)\,\Gamma(4n+3-c)\,\Gamma(4n+3-d)\,\Gamma(3n+2-c-d)}
\nonumber\\ &\quad\qquad\times
{}_6F_5\biggl[\begin{matrix} 2n+1, \, n+\frac32, \, \frac c2, \, \frac c2+\frac12, \, \frac d2, \, \frac d2+\frac12 \\[2pt]
n+\frac12, \, 2n+2-\frac c2, \, 2n+\frac32-\frac c2, \, 2n+2-\frac d2, \, 2n+\frac32-\frac d2 \end{matrix}; 1\biggr].
\label{id-cat}
\end{align}
\end{theorem}

\begin{proof}
We start with Rahman's quadratic transformation \cite[Eq.~(7.8), $q\to 1$, reversed]{RV93}
\begin{align*}
&
{}_8F_7\biggl[\begin{matrix}
2 a - e, \, 1 + a - \frac{e}{2}, \, \frac{1}{2} + a - e, \, c, \, d, \, e, \\[1.5pt]
a - \frac{e}{2}, \, \frac{1}{2} + a, \, 1 + 2 a - c - e, \, 1 + 2 a - d - e, \, 1 + 2 a - 2 e,
\end{matrix}
\\ &\qquad\qquad\qquad\qquad\qquad\qquad
\begin{matrix}
1 + 4 a - c - d - e + n, \, -n \\[1.5pt]
-2 a + c + d - n, \, 1 + 2 a - e + n
\end{matrix}\, ; 1\biggr]
\displaybreak[2]\\ &\quad
=\frac{(1 + 2 a - c)_n\, (1 + 2 a - d)_n\, (1 + 2 a - e)_n\, (1 + 2 a - c - d - e)_n}
{(1 + 2 a)_n\, (1 + 2 a - c - d)_n\, (1 + 2 a - c - e)_n\, (1 + 2 a - d - e)_n}
\\ &\quad\qquad \times
{}_{11}F_{10}\biggl[\begin{matrix}
a, \, 1 + \frac{a}{2}, \, e, \, \frac{c}{2}, \, \frac{1}{2} + \frac{c}{2}, \, \frac{d}{2}, \, \frac{1}{2} + \frac{d}{2}, \\[1.5pt]
\frac{a}{2}, \, 1 + a - e, \, 1 + a - \frac{c}{2}, \, \frac{1}{2} + a - \frac{c}{2}, \, 1 + a - \frac{d}{2}, \, \frac{1}{2} + a - \frac{d}{2},
\end{matrix}
\\ &\qquad\qquad\qquad
\begin{matrix}
\frac{1}{2} + 2 a - \frac{c}{2} - \frac{d}{2} - \frac{e}{2} + \frac{n}{2}, \,
1 + 2 a - \frac{c}{2} - \frac{d}{2} - \frac{e}{2} + \frac{n}{2}, \, \frac{1}{2} - \frac{n}{2}, \, -\frac{n}{2} \\[1.5pt]
\frac{1}{2} - a + \frac{c}{2} + \frac{d}{2} + \frac{e}{2} - \frac{n}{2}, \,
- a + \frac{c}{2} + \frac{d}{2} + \frac{e}{2} - \frac{n}{2}, \, \frac{1}{2} + a + \frac{n}{2}, \, 1 + a + \frac{n}{2}
\end{matrix} ; 1\biggr],
\end{align*}
in which we let $n$ tend to $\infty$:
\begin{align*}
&
{}_6F_5\biggl[\begin{matrix}
2 a - e, \, 1 + a - \frac{e}{2}, \, \frac{1}{2} + a - e, \, c, \, d, \, e \\[1.5pt]
a - \frac{e}{2}, \, \frac{1}{2} + a, \, 1 + 2 a - c - e, \, 1 + 2 a - d - e, \, 1 + 2 a - 2 e
\end{matrix}\, ; -1\biggr]
\\ &\quad
=\frac{\Ga(1 + 2 a) \, \Ga(1 + 2 a - c - d) \, \Ga(1 + 2 a - c - e) \, \Ga(1 + 2 a - d - e)}
{\Ga(1 + 2 a - c) \, \Ga(1 + 2 a - d) \, \Ga(1 + 2 a - e) \, \Ga(1 + 2 a - c - d - e)}
\\ &\quad\qquad \times
{}_7F_6\biggl[\begin{matrix}
a, \, 1 + \frac{a}{2}, \, e, \, \frac{c}{2}, \, \frac{1}{2} + \frac{c}{2}, \, \frac{d}{2}, \, \frac{1}{2} + \frac{d}{2} \\[1.5pt]
\frac{a}{2}, \, 1 + a - e, \, 1 + a - \frac{c}{2}, \, \frac{1}{2} + a - \frac{c}{2}, \, 1 + a - \frac{d}{2}, \, \frac{1}{2} + a - \frac{d}{2}
\end{matrix} ; 1\biggr].
\end{align*}
Now set $a=2n+1$ and $e=n+1$ to deduce \eqref{id-cat}.
\end{proof}

The corresponding $q$-version, which we record here for completeness, reads
\begin{align*}
&
\,{}_8\phi_7\biggl[\begin{matrix} q^{3n+1}, \, q^{\frac{3n}2+\frac32}, \, -q^{\frac{3n}2+\frac32}, \, q^{n+\frac12}, \, -q^{n+\frac12}, \, q^{n+1}, \, c, \, d \\[2pt]
q^{\frac{3n}2+\frac12}, \, -q^{\frac{3n}2+\frac12}, \, q^{2n+\frac32}, \, -q^{2n+\frac32}, \, q^{2n+1}, \, q^{3n+2}/c, \, q^{3n+2}/d \end{matrix}\,; q,-\frac{q^{3n+2}}{cd}\biggr]
\\ &\quad
=\frac{(q^{3n+2},q^{4n+3}/c,q^{4n+3}/d,q^{3n+2}/cd;q)_\infty}{(q^{4n+3},q^{3n+2}/c,q^{3n+2}/d,q^{4n+3}/cd;q)_\infty}
\\ &\quad\qquad\times
{}_7\phi_6\biggl[\begin{matrix} q^{4n+2}, \, q^{2n+3}, \, -q^{2n+3}, \, c, \, cq, \, d, \, dq \\[2pt]
q^{2n+1}, \, -q^{2n+1}, \, q^{4n+4}/c, \, q^{4n+3}/c, \, q^{4n+4}/d, \, q^{4n+3}/d \end{matrix}\,; q^2,\frac{q^{6n+4}}{c^2d^2}\biggr].
\end{align*}

\section{Logarithm of $2$}
\label{s3}

\null\hfill\smash{\raisebox{-12mm\relax}%
  {\includegraphics[scale=0.03]{log2}}}\strut \\[-15mm]

\parshape 5
  0pt 0.8\textwidth
  0pt 0.8\textwidth
  0pt 0.8\textwidth
  0pt 0.8\textwidth
  0pt \textwidth
Another strange identity is related to the classical rational approximations to $\log2$:
\begin{align*}
r_n
&=(-1)^{n+1}\sum_{t=0}^\infty\frac{\prod_{j=1}^n(t-j)}{\prod_{j=0}^n(t+j)}\,(-1)^t
\\
&=\frac{\Gamma(n+1)^2}{\Gamma(2n+2)}\,{}_2F_1\biggl[\begin{matrix} n+1, \, n+1 \\ 2n+2 \end{matrix}; -1\biggr]
\\
&=\int_0^1\frac{x^n(1-x)^n}{(1+x)^{n+1}}\,\d x.
\end{align*}
The sequence satisfies the recurrence equation $(n+1)r_{n+1}-3(2n+1)r_n+nr_{n-1}=0$, and with the help of the latter we find out that
$r_n=\wt r_n$ for
\begin{align*}
\wt r_n
&=\sum_{t=0}^\infty\frac{(2n+1)!\,\prod_{j=1}^n(t-j)}{n!\,\prod_{j=0}^{2n+1}(2t-n-1+j)}
\\
&=\frac{\Gamma(n+1)\,\Gamma(2n+2)}{\Gamma(3n+3)}\,{}_3F_2\biggl[\begin{matrix} n+1, \, \frac n2+\frac12, \, \frac n2+1 \\[2pt]
\frac{3n}2+2, \, \frac{3n}2+\frac32 \end{matrix}; 1\biggr].
\end{align*}
The finding is a particular case of another general identity.

\begin{theorem}
\label{th-ln2}
We have
\begin{equation}
{}_2F_1\biggl[\begin{matrix} x, \, 2a \\ 2b-x \end{matrix}; -1\biggr]
=\frac{\Gamma(2b-x)\,\Gamma(2b-2a)}{\Gamma(2b)\,\Gamma(2b-2a-x)}
\,{}_3F_2\biggl[\begin{matrix} x, \, a, \, a+\frac12 \\[2pt] b, \, b+\frac12 \end{matrix}; 1\biggr].
\label{id-ln2}
\end{equation}
\end{theorem}

\begin{proof}
This is a specialisation of a transformation of 
Whipple \cite[Sec.~4.6, Eq.~(3)]{Ba35}: set $b=\kappa-a$ there
and reparametrise.
\end{proof}

A companion $q$-version is
\begin{multline*}
{}_6\phi_7\biggl[\begin{matrix} -b/q, \, \sqrt{-bq}, \, -\sqrt{-bq}, \, x, \, -x, \, a \\
\sqrt{-b/q}, \, \sqrt{-b/q}, \, -b/x, \, b/x, \, -b/a, \, 0, \, 0 \end{matrix}\,; q,-\frac{b^2}{ax^2}\biggr]
\\
=\frac{(b^2,b^2/(ax)^2;q^2)_\infty}{(b^2/a^2,b^2/x^2;q^2)_\infty}
\,{}_3\phi_2\biggl[\begin{matrix} x^2, \, a, \, aq \\ b, \, bq \end{matrix}; q^2,\frac{b^2}{a^2x^2}\biggr],
\end{multline*}
which follows from \cite[Eq.~(3.10.4)]{GR04}.

To clarify the arithmetic situation behind the right-hand side of \eqref{id-ln2}, we notice that there is a permutation group for it
used for producing a sharp irrationality measure of $\zeta(2)$ in \cite{RV96}. As explained in \cite[Section~6]{Zu04},
a realisation of the group for a generic hypergeometric function
\begin{equation}
\frac{\Gamma(a_2)\,\Gamma(b_2-a_2)\,\Gamma(a_3)\,\Gamma(b_3-a_3)}{\Gamma(b_2)\,\Gamma(b_3)}
\,{}_3F_2\biggl[\begin{matrix} a_1, \, a_2, \, a_3 \\ b_2, \, b_3 \end{matrix}; 1\biggr]
\label{3F2}
\end{equation}
can be given by means of the ten parameters
\begin{align*}
c_{00}&=(b_2+b_3)-(a_1+a_2+a_3)-1,
\\
c_{jk}&=\begin{cases}
a_j-1, & \text{for }k=1, \\
b_k-a_j-1, & \text{for } k=2,3,
\end{cases}
\end{align*}
as follows. If the set of parameters is represented in the matrix form
\begin{equation}
\bc=\begin{pmatrix} c_{00} \\ & c_{11} & c_{12} & c_{13} \\ & c_{21} & c_{22} & c_{23} \\ & c_{31} & c_{32} & c_{33} \end{pmatrix},
\label{bc}
\end{equation}
and $H(\bc)$ denotes the corresponding hypergeometric function in \eqref{3F2}, then the quantity
\begin{equation}
\frac{H(\bc)}{\Gamma(c_{00}+1)\,\Gamma(c_{21}+1)\,\Gamma(c_{31}+1)\,\Gamma(c_{22}+1)\,\Gamma(c_{33}+1)}
\label{inv-H}
\end{equation}
is invariant under the group $\fG$ (of order 120) generated by the four involutions
\begin{alignat*}{2}
\fa_1&=(c_{11} \; c_{21})\,(c_{12} \; c_{22})\,(c_{13} \; c_{23}), &
\fa_2&=(c_{21} \; c_{31})\,(c_{22} \; c_{32})\,(c_{23} \; c_{33}), \\
\fb&=(c_{12} \; c_{13})\,(c_{22} \; c_{23})\,(c_{32} \; c_{33}), &\quad\text{and}\quad
\fh&=(c_{00} \; c_{22})\,(c_{11} \; c_{33})\,(c_{13} \; c_{31}).
\end{alignat*}
Notice that the permutations $\fa_1$, $\fa_2$, and $\fb$ correspond to the rearrangements $a_1\leftrightarrow a_2$, $a_2\leftrightarrow a_3$,
and $b_2\leftrightarrow b_3$, respectively, of the function \eqref{3F2}, so that the invariance of \eqref{inv-H} under their action
is trivial. It is only the permutation $\fh$, underlying Thomae's 
transformation
\cite[Sec.~3.2, Eq.~(1)]{Ba35} and Whipple's 
transformation \cite[Sec.~4.4, Eq.~(2)]{Ba35},
that makes the action of the group on \eqref{inv-H} non-trivial.

With the method in \cite[Section 3.3]{FZ10}, if
\begin{equation}
a_1,a_2,b_2\in\mathbb Z \quad\text{and}\quad a_3,b_3\in\mathbb Z+\tfrac12
\label{cond}
\end{equation}
are chosen such that $c_{jk}\ge-\frac12$ for all~$j$ and~$k$, 
then the quantity $H(\bc)$ representing~\eqref{3F2} satisfies
$$
H(\bc)\in\mathbb Q\log2+\mathbb Q.
$$
It is a tough task to produce a sharp integer $D(\bc)$ such that $D(\bc)H(\bc)\in\mathbb Z\log2+\mathbb Z$ in the general case;
it can be given in the particular situation
where $a_3-a_2=b_3-b_2=\pm\frac12$ with the help of \eqref{id-ln2} and
the known information for 
the corresponding ${}_2F_1(1)$-series.

Observe that the group $\fG=\langle\fa_1,\,\fa_2,\,\fb,\,\fh\rangle$ cannot be arithmetically used in its full force when
the parameters of \eqref{3F2} are subject to \eqref{cond}. However, apart from the initial representative \eqref{bc},
there are five more with the constraint that entries $1\,3$, $2\,3$, $3\,1$ and $3\,2$ are from $\mathbb Z+\frac12$, namely
\begin{equation}
\begin{gathered}
\begin{pmatrix} c_{22} \\ & c_{33} & c_{12} & c_{31} \\ & c_{21} & c_{00} & c_{23} \\ & c_{13} & c_{32} & c_{11} \end{pmatrix},
\quad
\begin{pmatrix} c_{12} \\ & c_{11} & c_{00} & c_{13} \\ & c_{33} & c_{22} & c_{31} \\ & c_{23} & c_{32} & c_{21} \end{pmatrix},
\quad
\begin{pmatrix} c_{33} \\ & c_{22} & c_{21} & c_{13} \\ & c_{12} & c_{11} & c_{23} \\ & c_{31} & c_{32} & c_{00} \end{pmatrix},
\\
\begin{pmatrix} c_{11} \\ & c_{00} & c_{21} & c_{31} \\ & c_{12} & c_{33} & c_{23} \\ & c_{13} & c_{32} & c_{22} \end{pmatrix},
\quad\text{and}\quad
\begin{pmatrix} c_{21} \\ & c_{22} & c_{33} & c_{13} \\ & c_{00} & c_{11} & c_{31} \\ & c_{23} & c_{32} & c_{12} \end{pmatrix},
\end{gathered}
\label{5bc}
\end{equation}
and another six which are obtained from \eqref{bc} and \eqref{5bc} by further action of $\fa_1$.

Remarkably enough, the choices $x=12n+1$, $a=14n+1$, $b=28n+2$ and $x=14n+1$, $a=12n+1$, $b=28n+2$ in \eqref{id-ln2}, which originate from the trivial transformation
of the $_2F_1(-1)$-side and which correspond to an early 
(`pre-Raffaele'~\cite{Ma09})
irrationality measure record \cite{Ha90,Ru87,Vi97,Zu04b}, produce $\fG$-disjoint collections
$$
{\left(
\begin{smallmatrix} 14n+1 \\ & 12n+1 & 8n+1 & 8n+\frac12 \\[1pt] & 7n+1 & 13n+1 & 13n+\frac12 \\[1pt] & 7n+\frac12 & 13n+\frac12 & 13n+1 \end{smallmatrix}\right)
}
\quad\text{and}\quad
{\left(
\begin{smallmatrix} 16n+1 \\ & 14n+1 & 7n+1 & 7n+\frac12 \\[1pt] & 6n+1 & 15n+1 & 15n+\frac12 \\[1pt] & 6n+\frac12 & 15n+\frac12 & 15n+1 \end{smallmatrix}\right)
}
$$
on the $_3F_2(1)$-side.

\section(\003\300 squared){$\pi$ squared}
\label{s4}

\null\hfill\smash{\raisebox{-12mm\relax}%
  {\includegraphics[scale=0.03]{pi-squared}}}\strut \\[-15mm]

\parshape 5
  0pt 0.8\textwidth
  0pt 0.8\textwidth
  0pt 0.8\textwidth
  0pt 0.8\textwidth
  0pt \textwidth
Our next hypergeometric entry \emph{a priori} produces linear forms not only in $1$ and $\zeta(2)=\pi^2/6$
but also in $\zeta(4)=\pi^4/90$, with rational coefficients.
It originates from the well-poised hypergeometric series
\begin{align*}
r_n
&=\sum_{t=1}^\infty\frac{2^{8n}n!^4(2n)!^2\prod_{j=0}^{4n-1}(t-n+j)}{(4n)!\,\prod_{j=0}^{2n}(t-\frac12+j)^4}
\\
&=\frac{\pi^2\Gamma(2n+1)^6}{\Gamma(3n+\frac32)^4}
\,{}_5F_4\biggl[\begin{matrix} 4n+1, \, n+\frac12, \, n+\frac12, \, n+\frac12, \, n+\frac12 \\[2pt]
3n+\frac32, \, 3n+\frac32, \, 3n+\frac32, \, 3n+\frac32 \end{matrix}; 1\biggr].
\end{align*}
It is standard to sum the rational function
$$
R_n(t)=\frac{2^{8n}n!^4(2n)!^2\prod_{j=0}^{4n-1}(t-n+j)}{(4n)!\,\prod_{j=0}^{2n}(t-\frac12+j)^4}
$$
by expanding it into the sum of partial fractions; the well-poised symmetry $R_n(t)=R_n(2n-1-t)$ (and the residue sum theorem) imply then that
$$
r_n\in\mathbb Q\pi^4+\mathbb Q\pi^2+\mathbb Q
$$
for $n=0,1,2,\dots$\,. At the same time,
$$
r_0=\frac16\,\pi^4,
\quad
r_1=\frac{19}{6}\,\pi^4-\frac{125}{4}\,\pi^2,
$$
and the sequence $r_n$ satisfies a second order recurrence equation, so that
$r_n=a_n\pi^4-b_n\pi^2\in\mathbb Q\pi^4+\mathbb Q\pi^2$ for all $n$.
This happens because the function $R_n(t)$ vanishes at $t=1,0,-1,\dots,-n+2$ so that
$$
r_n
=\sum_{t=-n+1}^\infty R_n(t),
$$
and in view of the following result.

\begin{lemma}
\label{lem1}
Assume that a rational function
$$
R(t)=\sum_{i=1}^s\sum_{k=0}^n\frac{a_{i,k}}{(t+k)^i}
$$
satisfies $R(t)=R(-n-t)$. Put $m=\lfloor(n-1)/2\rfloor$.
Then $a_{i,n-k}=(-1)^ia_{i,k}$ and
\begin{align*}
\sum_{t=-m}^\infty R(t-\tfrac12)
&=\sum_{\substack{i=2\\i\;\text{\em even}}}^sa_i\sum_{\ell=1}^\infty\frac1{(\ell-\frac12)^i}+a_0
\\
&=\sum_{\substack{i=2\\i\;\text{\em even}}}^sa_i(2^i-1)\zeta(i)+a_0
\in\mathbb Q+\mathbb Q\,\pi^2+\mathbb Q\,\pi^4+\dots+\mathbb Q\,\pi^{2\lfloor s/2\rfloor},
\end{align*}
where
$$
a_i=\sum_{k=0}^na_{i,k}, \quad\text{for }i=2,\dots,s,
\quad\text{and}\quad
a_0=\begin{cases}
0, &\text{for $n$ even}, \\
\tfrac12\,R(-m-\tfrac12), &\text{for $n$ odd}.
\end{cases}
$$
\end{lemma}

\begin{proof}
The property $a_{i,n-k}=(-1)^ia_{i,k}$ is straightforward
to see from $R(t)=R(-n-\nobreak t)$. Furthermore, we have
\begin{align*}
\sum_{t=-m}^\infty R(t-\tfrac12)
&=\sum_{i=1}^s\sum_{k=0}^na_{i,k}\sum_{t=-m}^\infty\frac1{(t+k-\frac12)^i}
=\sum_{i=1}^s\sum_{k=0}^na_{i,k}\sum_{\ell=k-m}^\infty\frac1{(\ell-\frac12)^i}
\\
&=\sum_{i=1}^s\sum_{k=0}^na_{i,k}\cdot\sum_{\ell=1}^\infty\frac1{(\ell-\frac12)^i}
\\ &\qquad
+\sum_{i=1}^s\biggl(\sum_{k=0}^ma_{i,k}\sum_{\ell=k-m}^0\frac1{(\ell-\frac12)^i}
-\sum_{k=m+1}^na_{i,k}\sum_{\ell=1}^{k-m-1}\frac1{(\ell-\frac12)^i}\biggr)
\\
&=\sum_{\substack{i=2\\i\;\text{even}}}^sa_i\sum_{\ell=1}^\infty\frac1{(\ell-\frac12)^i}
+a_0,
\end{align*}
with the constant term equal to
\begin{align*}
a_0
&=\sum_{i=1}^s\biggl(\sum_{k=0}^ma_{i,k}\sum_{\ell=k-m}^0\frac1{(\ell-\frac12)^i}
-\sum_{k=m+1}^na_{i,k}\sum_{\ell=1}^{k-m-1}\frac1{(\ell-\frac12)^i}\biggr)
\\
&=\sum_{i=1}^s(-1)^i\biggl(\sum_{k=0}^ma_{i,n-k}\sum_{\ell=k-m}^0\frac1{(\ell-\frac12)^i}
-\sum_{k=m+1}^na_{i,n-k}\sum_{\ell=1}^{k-m-1}\frac1{(\ell-\frac12)^i}\biggr)
\\ \intertext{(take $k'=n-k$)}
&=\sum_{i=1}^s(-1)^i\biggl(\sum_{k'=n-m}^na_{i,k'}\sum_{\ell=n-m-k'}^0\frac1{(\ell-\frac12)^i}
-\sum_{k'=0}^{n-m-1}a_{i,k'}\sum_{\ell=1}^{n-m-1-k'}\frac1{(\ell-\frac12)^i}\biggr).
\end{align*}
If $n$ is odd, then $n-m=m+1$ and
\begin{align*}
a_0
&=\sum_{i=1}^s(-1)^i\biggl(\sum_{k'=m+1}^na_{i,k'}\sum_{\ell=m+1-k'}^0\frac1{(\ell-\frac12)^i}
-\sum_{k'=0}^ma_{i,k'}\sum_{\ell=1}^{m-k'}\frac1{(\ell-\frac12)^i}\biggr)
\\
&=\sum_{i=1}^s\biggl(\sum_{k'=m+1}^na_{i,k'}\sum_{\ell=1}^{k'-m}\frac1{(\ell-\frac12)^i}
-\sum_{k'=0}^ma_{i,k'}\sum_{\ell=k'-m+1}^0\frac1{(\ell-\frac12)^i}\biggr)
\\
&=\sum_{i=1}^s\biggl(\sum_{k'=m+1}^na_{i,k'}
\sum_{\ell=1}^{k'-m-1}\frac1{(\ell-\frac12)^i}
+\sum_{k'=m+1}^na_{i,k'}\frac1{(k'-m-\frac12)^i}
\\ &\qquad
-\sum_{k'=0}^ma_{i,k'}\sum_{\ell=k'-m}^0\frac1{(\ell-\frac12)^i}
+\sum_{k'=0}^ma_{i,k'}\frac1{(k'-m-\frac12)^i}\biggr)
\\
&=-a_0+\sum_{i=1}^s\sum_{k'=0}^n\frac{a_{i,k'}}{(k'-m-\frac12)^i}
=-a_0+R(-m-\tfrac12).
\end{align*}
Similarly, if $n$ is even, then $n-m=m+2$ and
\begin{align*}
a_0
&=\sum_{i=1}^s(-1)^i\biggl(\sum_{k'=m+2}^na_{i,k'}\sum_{\ell=m+2-k'}^0\frac1{(\ell-\frac12)^i}
-\sum_{k'=0}^{m+1}a_{i,k'}\sum_{\ell=1}^{m+1-k'}\frac1{(\ell-\frac12)^i}\biggr)
\\
&=\sum_{i=1}^s\biggl(\sum_{k'=m+2}^na_{i,k'}\sum_{\ell=1}^{k'-m-1}\frac1{(\ell-\frac12)^i}
-\sum_{k'=0}^{m+1}a_{i,k'}\sum_{\ell=k'-m}^0\frac1{(\ell-\frac12)^i}\biggr)
\\
&=\sum_{i=1}^s\biggl(\sum_{k'=m+1}^na_{i,k'}\sum_{\ell=1}^{k'-m-1}\frac1{(\ell-\frac12)^i}
-\sum_{k'=0}^ma_{i,k'}\sum_{\ell=k'-m}^0\frac1{(\ell-\frac12)^i}\biggr)
\\
&=-a_0.
\end{align*}
This implies the required formula for $a_0$.
\end{proof}

The characteristic polynomial of the recursion for $r_n$ as above is $\lambda^2-123\lambda+1$, and its zeroes are
quite recognisable: $((1\pm\sqrt5)/2)^{10}$. After performing some
experiments, it turns out that
$$
r_n
=\frac{\pi^2(2n)!^4}{(4n+1)!^2}
\,{}_3F_2\biggl[\begin{matrix} 2n+1, \, 2n+1, \, 2n+1 \\ 4n+2, \, 4n+2 \end{matrix}; 1\biggr],
$$
where the latter is a `rarified' sequence of the Ap\'ery approximations to $\zeta(2)$. This follows as a consequence
of the hypergeometric identity
\begin{multline} \label{eq:2n+1/4n+2}
\frac{\Gamma(4n+2)^2\Gamma(2n+1)^2}{\Gamma(3n+\frac32)^4}
\,{}_5F_4\biggl[\begin{matrix} 4n+1, \, n+\frac12, \, n+\frac12, \, n+\frac12, \, n+\frac12 \\[2pt]
3n+\frac32, \, 3n+\frac32, \, 3n+\frac32, \, 3n+\frac32 \end{matrix};1\biggr]
\\
={}_3F_2\biggl[\begin{matrix} 2n+1, \, 2n+1, \, 2n+1 \\ 4n+2, \, 4n+2 \end{matrix};1\biggr],
\end{multline}
which is in turn the particular case where $a=b=c=2n+1$ of the following general transformation.

\begin{theorem}
\label{th-pi2}
We have
\begin{align}
&
{} _{3} F _{2} \biggl[ \begin{matrix} a, \, b, \, c \\ \, -a+2 b+c, \, -a+b+2 c \end{matrix}\,; 1\biggr]
=\frac {\Ga(-\frac{a}{2}+b+c+\frac{1}{2})\, \Ga(-\frac{3 a}{2}+2 b+c+\frac{1}{2})}
{\Ga(-a+b+c+\frac{1}{2})\, \Ga(-a+2 b+c+\frac{1}{2})}
\nonumber\\ &\qquad\times
\frac { \Ga(-a+b+2 c)\, \Ga(-3 a+2 b+2 c) \,\Ga(-2 a+2 b+2 c) \,\Ga(-2 a+4 b+2 c)}
{\Ga(-\frac{3 a}{2}+b+2 c)\, \Ga(-\frac{5 a}{2}+2 b+2 c)\, \Ga(-a+2 b+2 c)\, \Ga(-3 a+4 b+2 c)}
\nonumber\\ &\qquad\times
 {}_5{F}_4 \biggl[ \begin{matrix} 
-2 a+2 b+2 c-1, \, c-\frac{a}{2}, \, -\frac{3 a}{2}+b+c, \, \frac{a}{2}, \, b-\frac{a}{2} \\[1.5pt]
 -\frac{3 a}{2}+2 b+c, \, -\frac{a}{2}+b+c, \, -\frac{5a}{2}+2 b+2 c, \, -\frac{3 a}{2}+b+2 c
\end{matrix}\, ; 1\biggr].
\label{eq:3F2-5F4}
\end{align}
\end{theorem}

\begin{proof}
We start with the
transformation formula (cf.\ \cite[Eq.~(3.5.10), $q\to 1$, reversed]{GR04})
\begin{multline} \label{eq:T3240}
{} _{3} F _{2} \biggl[ \begin{matrix} a, \, b, \, c \\ d, \, d - b + c \end{matrix}\, ;
   1\biggr]  = 
\frac {\Gamma( 2 d)\,\Gamma( 2 d - 2 b - a)\,\Gamma( d - b +
    c)\,\Gamma( d - a + c)} {\Gamma( 2 d - 2 b)\,\Gamma( 2 d - a)\,\Gamma( d + c)\,
   \Gamma( d - b - a + c)}
\\ \times
{}_7{F}_6 \biggl[ \begin{matrix} 
d-\frac{1}{2}, \, \frac{d}{2}+\frac{3}{4}, \, \frac{d}{2}-\frac{c}{2}, \, b, \, \frac{a}{2}, \, \frac{a}{2}+\frac{1}{2}, \, -\frac{c}{2}+\frac{d}{2}+\frac{1}{2} \\[1.5pt]
\frac{d}{2}-\frac{1}{4}, \,\frac{c}{2}+\frac{d}{2}+\frac{1}{2}, \, -b+d+\frac{1}{2}, \, -\frac{a}{2}+d+\frac{1}{2}, \, d-\frac{a}{2}, \, \frac{c}{2}+\frac{d}{2} 
                 \end{matrix}
                 ; 1\biggr].
\end{multline}
To the very-well-poised $_7F_6$-series on the right-hand side we apply
the transformation formula (cf.\ \cite[Sec.~7.5, Eq.~(2)]{Ba35})
\begin{align}
&
{} _{7} F _{6} \biggl[\begin{matrix} 
a, \, \frac{a}{2}+1, \, b, \, c, \, d, \, e, \, f \\[1.5pt]
 \frac{a}{2}, \, a-b+1, \, a-c+1, \, a-d+1, \, a-e+1, \, a-f+1\end{matrix};1 \biggr]
\nonumber\displaybreak[2]\\ &\quad
=\frac{\Ga(a-c+1) \,\Ga(a-d+1)\, \Ga(a-e+1)\, \Ga(a-f+1)}
{\Ga(a+1)\, \Ga(b)\, \Ga(2 a-b-c-d-e+2)\, \Ga(2 a-b-c-d-f+2)}
\nonumber\\ &\quad\qquad\kern-8pt 
\times
\frac {\Ga(3 a-2 b-c-d-e-f+3) \,\Ga(2 a-b-c-d-e-f+2)} 
{ \Ga(2 a-b-c-e-f+2)\, \Ga(2 a-b-d-e-f+2)}
\nonumber\displaybreak[2]\\ &\quad\qquad\kern-8pt 
\times
{}_7F_6 \biggl[\begin{matrix} 
3 a-2 b-c-d-e-f+2, \, \frac{3a}{2}-b-\frac{c}{2}-\frac{d}{2}-\frac{e}{2}-\frac{f}{2}+2, \, a-b-c+1, \\[1.5pt]
\frac{3a}{2}-b-\frac{c}{2}-\frac{d}{2}-\frac{e}{2}-\frac{f}{2}+1, \, 2 a-b-d-e-f+2,
\end{matrix}
\nonumber\\ &\quad\qquad \qquad\qquad\qquad
\begin{matrix}
a-b-d+1, \, a-b-e+1, \\
2 a-b-c-e-f+2, \, 2 a-b-c-d-f+2,
\end{matrix}
\nonumber\\ &\quad\qquad \qquad\qquad\qquad\qquad\qquad
\begin{matrix}
a-b-f+1, \, 2 a-b-c-d-e-f+2\\
2 a-b-c-d-e+2, \, a-b+1
\end{matrix};1 \biggr].
\label{eq:T7635}
\end{align}
Thus, we obtain
\begin{align*}
&
{} _{3} F _{2} \biggl[ \begin{matrix} a, \, b, \, c \\ d, \, d - b + c \end{matrix} \,; 1\biggr]
= \frac {\Ga(2 d)\, \Ga(d-\frac{a}{2})\, \Ga(-\frac{a}{2}+d+\frac{1}{2})\, \Ga(-b+d+\frac{1}{2})\, \Ga(\frac{c}{2}+\frac{d}{2})}
{\Ga(d+\frac{1}{2}) \,\Ga(2 d-a) \,\Ga(2 d-2 b)\, \Ga(\frac{d}{2}-\frac{c}{2})\, \Ga(c+d)}
\\ &\qquad\times
\frac {\Ga(-a-2 b+2 d)\, \Ga(-b+c+d) \, \Ga(-a-b+\frac{3 c}{2}+\frac{3 d}{2}+\frac{1}{2})}
{\Ga(-\frac{a}{2}-b+c+d)\, \Ga(-\frac{a}{2}-b+c+d+\frac{1}{2})\, \Ga(-a-b+\frac{c}{2}+\frac{3 d}{2}+\frac{1}{2})}
\displaybreak[2]\\ &\qquad\times
{} _{7} F _{6} \biggl[ \begin{matrix}
-a-b+\frac{3 c}{2}+\frac{3 d}{2}-\frac{1}{2}, \, -\frac{a}{2}-\frac{b}{2}+\frac{3 c}{4}+\frac{3 d}{4}+\frac{3}{4}, \,
-\frac{a}{2}+\frac{c}{2}+\frac{d}{2}, \, c,
\\[1.5pt]
-\frac{a}{2}-\frac{b}{2}+\frac{3 c}{4}+\frac{3 d}{4}-\frac{1}{4}, \,
-\frac{a}{2}-b+c+d+\frac{1}{2}, \, -a-b+\frac{c}{2}+\frac{3 d}{2}+\frac{1}{2},
\end{matrix}
\\ &\qquad\qquad\qquad\qquad
\begin{matrix}
-a-b+c+d, \, -b+\frac{c}{2}+\frac{d}{2}+\frac{1}{2}, \,  -\frac{a}{2}+\frac{c}{2}+\frac{d}{2}+\frac{1}{2} \\[1.5pt]
\frac{c}{2}+\frac{d}{2}+\frac{1}{2}, \, -a+c+d, \, -\frac{a}{2}-b+c+d
\end{matrix} ; 1 \biggr] .
\end{align*}
Next we apply the transformation formula (cf.\ \cite[Sec.~7.5, Eq.~(1)]{Ba35})
\begin{align}
&
{} _{7} F _{6} \biggl[\begin{matrix} 
a, \, \frac{a}{2}+1, \, b, \, c, \, d, \, e, \, f \\[1.5pt]
 \frac{a}{2}, \, a-b+1, \, a-c+1, \, a-d+1, \, a-e+1, \, a-f+1\end{matrix};1 \biggr]
\nonumber\displaybreak[2]\\ &\quad
=\frac {\Ga(a-e+1) \, \Ga(a-f+1) \, \Ga(2 a-b-c-d+2) \, \Ga(2 a-b-c-d-e-f+2) } 
{\Ga(a+1) \, \Ga(a-e-f+1) \, \Ga(2 a-b-c-d-e+2) \, \Ga(2 a-b-c-d-f+2)}
\nonumber\displaybreak[2]\\ &\quad\qquad \times
{} _{7} F _{6} \biggl[ \begin{matrix}
2 a-b-c-d+1, \, a-\frac{b}{2}-\frac{c}{2}-\frac{d}{2}+\frac{3}{2}, \, a-c-d+1, \, a-b-d+1, \\[1.5pt]
a-\frac{b}{2}-\frac{c}{2}-\frac{d}{2}+\frac{1}{2}, \, a-b+1,\, a-c+1,
\end{matrix}
\nonumber\\ &\qquad\qquad\qquad\qquad
\begin{matrix}
a-b-c+1, \, e, \, f \\
a-d+1, \, 2 a-b-c-d-e+2, \, 2 a-b-c-d-f+2
\end{matrix}\,;  1\biggr] .
\label{eq:T7634}
\end{align}
We arrive at
\begin{align*}
&
{} _{3} F _{2} \biggl[ \begin{matrix} a, \, b, \, c \\ d, \, d - b + c \end{matrix} \,; 1\biggr]
=\frac {\Ga(2 d)\, \Ga(d-\frac{a}{2})\, \Ga(\frac{c}{2}+\frac{d}{2})\, \Ga(-a-2 b+2 d)} 
{\Ga(d+\frac{1}{2}) \,\Ga(2 d-a)\, \Ga(2 d-2 b)\, \Ga(c+d)}
\\ &\qquad\times
\frac{\Ga(-a+c+d)\, \Ga(-b+c+d)\, \Ga(-\frac{a}{2}-b+\frac{c}{2}+\frac{3 d}{2}+1)}
{\Ga(-\frac{a}{2}+\frac{c}{2}+\frac{d}{2}-\frac{1}{2})\, \Ga(-\frac{a}{2}-b+c+d+\frac{1}{2})\,
\Ga(-a-b+\frac{c}{2}+\frac{3 d}{2}+\frac{1}{2})}
\displaybreak[2]\\ &\qquad\times
{} _{7} F _{6} \biggl[ \begin{matrix}
-\frac{a}{2}-b+\frac{c}{2}+\frac{3 d}{2}, \, -\frac{a}{4}-\frac{b}{2}+\frac{c}{4}+\frac{3 d}{4}+1, \,
-\frac{a}{2}+\frac{c}{2}+\frac{d}{2}+\frac{1}{2}, \, -\frac{c}{2}+\frac{d}{2}+\frac{1}{2}, \\[1.5pt]
-\frac{a}{4}-\frac{b}{2}+\frac{c}{4}+\frac{3 d}{4}, \, -b+d+\frac{1}{2}, \, -\frac{a}{2}-b+c+d+\frac{1}{2},
\end{matrix}
\\ &\qquad\qquad\qquad\qquad\qquad\qquad
\begin{matrix}
\frac{a}{2}+\frac{1}{2}, \, -\frac{a}{2}-b+d+\frac{1}{2}, \, -b+\frac{c}{2}+\frac{d}{2}+\frac{1}{2} \\[1.5pt]
-a-b+\frac{c}{2}+\frac{3 d}{2}+\frac{1}{2}, \, \frac{c}{2}+\frac{d}{2}+\frac{1}{2}, \, -\frac{a}{2}+d+\frac{1}{2}
\end{matrix} ; 1\biggr] .
\end{align*}
Now we apply again \eqref{eq:T7635}. As a result, we obtain
\begin{align}
&
{} _{3} F _{2} \biggl[ \begin{matrix} a, \, b, \, c \\ d, \, d - b + c \end{matrix} \,; 1\biggr]
=\frac {\Ga(2 d)\, \Ga(-\frac{a}{2}+d+\frac{1}{2})\, \Ga(\frac{c}{2}+\frac{d}{2}+\frac{1}{2})\, \Ga(-a-2 b+2 d)} 
{\Ga(d+\frac{1}{2}) \,\Ga(2 d-a)\, \Ga(2 d-2 b) \,\Ga(c+d) }
\nonumber\\ &\qquad\times
\frac {\Ga(-a+c+d) \,\Ga(-b+c+d)\, \Ga(-\frac{a}{2}-b+\frac{c}{2}+\frac{3 d}{2})}
{\Ga(-\frac{a}{2}+\frac{c}{2}+\frac{d}{2}+\frac{1}{2})\, \Ga(-\frac{a}{2}-b+c+d)\, \Ga(-a-b+\frac{c}{2}+\frac{3 d}{2})}
\nonumber\displaybreak[2]\\ &\qquad\times
{} _{7} F _{6} \biggl[ \begin{matrix}
-\frac{a}{2}-b+\frac{c}{2}+\frac{3 d}{2}-1, \, -\frac{a}{4}-\frac{b}{2}+\frac{c}{4}+\frac{3 d}{4}+\frac{1}{2}, \,
-b+\frac{c}{2}+\frac{d}{2}, \, -\frac{a}{2}-b+d, \\[1.5pt]
-\frac{a}{4}-\frac{b}{2}+\frac{c}{4}+\frac{3 d}{4}-\frac{1}{2}, \, d-\frac{a}{2}, \,  \frac{c}{2}+\frac{d}{2},
\end{matrix}
\nonumber\\ &\qquad\qquad\qquad\qquad\qquad\qquad
\begin{matrix}
\frac{a}{2}, \, \frac{d}{2}-\frac{c}{2}, \, -\frac{a}{2}+\frac{c}{2}+\frac{d}{2}-\frac{1}{2} \\[1.5pt]
-a-b+\frac{c}{2}+\frac{3 d}{2}, \, -\frac{a}{2}-b+c+d, \, -b+d+\frac{1}{2} 
\end{matrix} ; 1\biggr].
\label{eq:3F2-7F6}
\end{align}
Here, we equate the second upper parameter and the last lower
parameter in the $_7F_6$-series, that is,
$$
\textstyle
-\frac{a}{4}-\frac{b}{2}+\frac{c}{4}+\frac{3 d}{4}+\frac{1}{2}
=-b+d+\frac{1}{2} ,
$$
or, equivalently, $d=c+2b-a$. If we make this substitution in
\eqref{eq:3F2-7F6}, then the $_7F_6$-series reduces to a
$_5F_4$-series. The corresponding transformation formula is \eqref{eq:3F2-5F4}.
\end{proof}

$q$-Analogues of \eqref{eq:2n+1/4n+2}, \eqref{eq:3F2-5F4}, and \eqref{eq:3F2-7F6} can be
obtained by going through the analogous computations when using 
the $_8\phi_7$-transformation formula \cite[Eq.~(3.5.10)]{GR04} 
instead of \eqref{eq:T3240}, the $_8\phi_7$-transformation formula
\cite[Appendix (III.24)]{GR04} instead of \eqref{eq:T7635},
and the $_8\phi_7$-transformation formula \cite[Eq.~(2.10.1)]{GR04} 
instead of \eqref{eq:T7634}. The $q$-analogue of
\eqref{eq:3F2-7F6} obtained in this way is
\begin{multline} \label{eq:3F2-7F6Q}
{} _{8} \phi _{7} \biggl[ \begin{matrix} 
\def\frac#1#2{#1 / #2}  -cd/q, i \sqrt{c d q}, -i
   \sqrt{c d q}, b, -b, c, -c, a
\\
i\sqrt{c d/q},
    -i\sqrt{c d/q}, -c d/b, c d/b, -d, d,
    -c d/a
\end{matrix} ; q,
   \frac {d^2} {ab^2}\biggr] 
\\
=\frac {\def\frac#1#2{#1 / #2}  
(d^2 q, \frac{c d q}{a}, \frac{c d^3}{a^2
    b^2},
                     \frac{c^2 d^2}{a b^2};
                     q^2)_\infty\,
            (\frac{d^2}{b^2}, \frac{d^2}{a},
                     c d, -c d, \frac{c d}{a b},
                     -\frac{c d}{a b};
                     q)_\infty}
    {\def\frac#1#2{#1 / #2}  
(\frac{c^2 d^2}{a^2
                     b^2}, c d q, \frac{d^2 q}{a},
                   \frac{c d^3}{a b^2};
                     q^2)_\infty
\,
(d^2, \frac{d^2}{a b^2}, \frac{c
                     d}{a}, -\frac{c d}{a}, \frac{c
                     d}{b}, -\frac{c
                     d}{b}; q)_\infty}
\\
\times
{} _{8} \phi _{7} \biggl[ \begin{matrix} 
\def\frac#1#2{#1 / #2} 
\frac{c d^3}{a b^2 q^2}, \sqrt{\frac{c d^3 q^2}{a b^2 }},
                     - \sqrt{\frac{c d^3 q^2}{a b^2 }},
                     \frac{c d}{b^2},
                     \frac{d^2}{a b^2}, a,
                     \frac{d}{c}, \frac{c d}{a q}\\
\def\frac#1#2{#1 / #2} 
\sqrt{\frac{c d^3 }{a b^2 q^2}}, - \sqrt{\frac{c d^3 }{a b^2 q^2}},
 \frac{d^2}{a}, c d,
                   \frac{c d^3}{a^2 b^2},
                     \frac{c^2 d^2}{a b^2}, \frac{d^2 q}{b^2}
\end{matrix} ; q^2,
   \frac {cdq} {a}\biggr] . 
\end{multline}
Similarly to before, we equate the second upper parameter and the
last lower parameter in the $_8\phi_7$-series on the right-hand side,
that is, 
$$
\sqrt{\frac{c d^3 q^2}{a b^2 }}=\frac{d^2 q}{b^2},
$$
or, equivalently, $d=cb^2/a$. If we substitute this in
\eqref{eq:3F2-7F6Q}, then we obtain
\begin{align*} 
&
{} _{8} \phi _{7} \biggr[ \begin{matrix} 
\def\frac#1#2{#1 / #2}  
 -\frac{b^2 c^2}{a q}, \frac{i b c \sqrt{q}}{\sqrt{a}},
 -\frac{i b c
                     \sqrt{q}}{\sqrt{a}}, b, -b,
                   c, -c, a \\
\def\frac#1#2{#1 / #2}  
\frac{i b c}{\sqrt{a} \sqrt{q}}, -\frac{i b c}{\sqrt{a}
  \sqrt{q}},
                     -\frac{b c^2}{a}, \frac{b c^2}{a},
                     -\frac{b^2 c}{a},
                     \frac{b^2 c}{a}, -\frac{b^2 c^2}{a^2}
\end{matrix} ; q,
   \frac {b^2 c^2} {a^3}\biggr]
\displaybreak[2]\\ &\quad
=
\frac {\def\frac#1#2{#1 / #2}  
(\frac{b^4 c^2 q}{a^2}, \frac{b^2 c^2 q}{a^2},
  \frac{b^4
                     c^4}{a^5}, \frac{b^2 c^4}{a^3};
                     q^2)_\infty}
    {\def\frac#1#2{#1 / #2}  
(\frac{b^2 c^4}{a^4}, \frac{b^2 c^2
                     q}{a}, \frac{b^4 c^2 q}{a^3},
                   \frac{b^4 c^4}{a^4};
                     q^2)_\infty}
\\ &\quad\qquad
\times
\frac {\def\frac#1#2{#1 / #2}  
 (\frac{b^4 c^2}{a^3}, \frac{b^2
                     c^2}{a}, -\frac{b^2 c^2}{a},
                   \frac{b c^2}{a^2}, -\frac{b
                     c^2}{a^2}; q)_\infty}
    {\def\frac#1#2{#1 / #2}  
   (\frac{b^4 c^2}{a^2},
                     \frac{b^2 c^2}{a^3}, -\frac{b^2
                       c^2}{a^2}, \frac{b
                     c^2}{a}, -\frac{b c^2}{a}; q)_\infty}
\\ &\quad\qquad
\times
{} _{6} \phi _{5} \biggl[ \begin{matrix} 
\def\frac#1#2{#1 / #2}  
\frac{b^4 c^4}{a^4 q^2}, -\frac{b^2 c^2 q}{a^2},
                     \frac{c^2}{a}, \frac{b^2 c^2}{a^3},
                     a, \frac{b^2}{a}\\
\def\frac#1#2{#1 / #2}  
-\frac{b^2 c^2}{a^2 q}, \frac{b^4 c^2}{a^3}, \frac{b^2
                     c^2}{a}, \frac{b^4 c^4}{a^5},
                   \frac{b^2 c^4}{a^3} 
\end{matrix} ; q^2,
   \frac {b^2 c^2 q} {a^2}\biggr]  ,
\end{align*}
a $q$-analogue of \eqref{eq:3F2-5F4}. 
Setting all of $a,b,c$ equal to $q^{2n+1}$, we arrive at a
$q$-analogue of \eqref{eq:2n+1/4n+2}, namely
\begin{align*} 
&
{} _{8} \phi _{7} \biggl[ \begin{matrix} 
\def\frac#1#2{#1 / #2}  
 -q^{6 n+2}, i q^{3 n+2}, -i q^{3 n+2}, q^{2
   n+1},
                     -q^{2 n+1}, q^{2 n+1}, -q^{2
                       n+1}, q^{2 n+1} \\
\def\frac#1#2{#1 / #2}  
i q^{3 n+1}, -i q^{3 n+1}, -q^{4 n+2}, q^{4
  n+2},
                     -q^{4 n+2}, q^{4 n+2}, -q^{4 n+2} 
\end{matrix} ; q,
   q^{2n+1}\biggr]
\\ &\quad
=
\frac {(q^{6 n+3}, q^{6 n+3}, -q^{6 n+3}, -q^{2
    n+1};
                     q)_\infty\,
   (q^{8 n+5}, q^{4
                     n+3}, q^{6 n+3}, q^{6 n+3};
                     q^2)_\infty }
    {(q^{8 n+4}, -q^{4 n+2}, q^{4 n+2},
                     -q^{4 n+2}; q)_\infty\,
   (q^{4 n+2}, q^{6 n+4}, q^{6
                     n+4}, q^{8 n+4}; q^2)_\infty}
\\ &\quad\qquad
\times
{} _{6} \phi _{5} \biggl [ \begin{matrix} 
\def\frac#1#2{#1 / #2}  
q^{8 n+2}, -q^{4 n+3}, q^{2 n+1}, q^{2
  n+1}, q^{2
                     n+1}, q^{2 n+1} \\
 -q^{4 n+1}, q^{6 n+3}, q^{6 n+3}, q^{6
   n+3}, q^{6
                     n+3} 
\end{matrix} ; q^2,
   q^{4n+3}\biggr]  .
\end{align*}

\section{Zeta values}
\label{s5}

In this section we prove Theorem~\ref{th-zeta}.

\begin{proof}[Proof of Theorem~\textup{\ref{th-zeta}}]
Fix an even integer $s\ge8$ and define the rational functions
\begin{align*}
R(t)=R_n(t)
&=\frac{n!^{s-6}\cdot2^{12n+1}(t+\frac n2)\prod_{j=1}^{3n}(t-n-\frac12+j)^2}{\prod_{j=0}^n(t+j)^s},
\\
\wh R(t)=\wh R_n(t)
&=\frac{n!^{s-6}\cdot2^{12n}\prod_{j=1}^{3n}(t-n-\frac12+j)^2}{\prod_{j=0}^n(t+j)^s},
\end{align*}
both vanishing together with their derivatives at $t=\nu-n+\frac12$ for $\nu=0,1,\dots,3n-1$. Then Lemma~\ref{lem1}
and the results from \cite[Section 2]{Zu18} apply, and we obtain the linear forms
\begin{align*}
r_n&=\sum_{\nu=1}^\infty R_n(\nu-\tfrac12)
=\sum_{\substack{i=2\\i\;\text{odd}}}^sa_i(2^i-1)\zeta(i)+a_0,
\\
r_n'&=-\sum_{\nu=1}^\infty\frac{\d R_n}{\d t}(\nu-\tfrac12)
=\sum_{\substack{i=2\\i\;\text{odd}}}^sa_ii(2^{i+1}-1)\zeta(i+1),
\displaybreak[2]\\
\wh r_n&=\sum_{\nu=1}^\infty\wh R_n(\nu-\tfrac12)
=\sum_{\substack{i=2\\i\;\text{even}}}^s\hat a_i(2^i-1)\zeta(i),
\\
\wh r_n'&=-\sum_{\nu=1}^\infty\frac{\d\wh R_n}{\d t}(\nu-\tfrac12)
=\sum_{\substack{i=2\\i\;\text{even}}}^s\hat a_ii(2^{i+1}-1)\zeta(i+1)+\hat a_0,
\end{align*}
with the following inclusions available:
$$
d_n^{s-i}a_i,\; d_n^{s-i}\hat a_i\in\mathbb Z \quad\text{for}\; i=2,3,\dots,s,
\quad\text{and}\quad
d_n^sa_0,\; d_n^{s+1}\hat a_0\in\mathbb Z.
$$
Here $d_n$ denotes the least common multiple of
$1,\dots,n$. Its asymptotic behaviour $d_n^{1/n}\to e$ as $n\to\infty$ 
follows from the prime number theorem.

The standard asymptotic machinery \cite[Section 2]{Zu02a} implies that
$$
\lim_{n\to\infty}|r_n|^{1/n}
=\lim_{n\to\infty}|\wh r_n|^{1/n}
=g(x_0)
$$
and
$$
\lim_{n\to\infty}|r_n'|^{1/n}
=\lim_{n\to\infty}|\wh r_n'|^{1/n}
=g(x_0'),
$$
where
$$
g(x)=\frac{2^{12}(x+3)^6(x+1)^s}{(x+2)^{2s}},
$$
and $x_0$, $x_0'$ are the real zeroes of the polynomial
$$
x^2(x+2)^s-(x+3)^2(x+1)^s
$$
on the intervals $x>0$ and $-1<x<0$, respectively. It can also be
observed numerically for each choice of even $s$ that 
$0<g(x_0')<g(x_0)$, so that
$$
\lim_{n\to\infty}|r_n-\mu r_n'|^{1/n}
=\lim_{n\to\infty}|\wh r_n-\hat\mu\wh r_n'|^{1/n}
=g(x_0)
$$
for any real $\mu$ and $\wh\mu$. Theorem~\ref{th-zeta} follows from taking $\mu=\lambda/\pi$ for the first collection,
$\hat\mu=4\lambda\pi$ for the second one, and noticing that, when $s=40$, we obtain
$$
g(x_0)=\exp(-40.54232882\dots)
\quad\text{and}\quad
g(x_0')=\exp(-40.54234026\dots),
$$
while for $s=42$ we get
\begin{equation*}
g(x_0)=\exp(-43.31492040\dots)
\quad\text{and}\quad
g(x_0')=\exp(-43.31492612\dots).
\qedhere
\end{equation*}
\end{proof}

Finally, we remark that further variations on Theorem~\ref{th-zeta} are possible by combining the two hypergeometric
constructions from this section and \cite{Zu18} (see also related applications in \cite{FSZ18} and \cite{RZ18}).
As the corresponding results remain similar in spirit to the theorem, we do not pursue this line here.


\end{document}